\documentclass[%
  reprint,
  amsmath,amssymb,
  aps,
  prd,
]{revtex4-2}

\usepackage{graphicx}
\usepackage{dcolumn}
\usepackage{bm}
\usepackage{amsmath,amsthm,amssymb}
\usepackage{enumitem}

\def\be{\begin{equation}}
\def\ee{\end{equation}}

\def\ra{\rightarrow}

\newtheorem*{theorem*}{Theorem}
\newtheorem*{conj*}{Conjecture}

\theoremstyle{remark}
\newtheorem{rem}{Remark}
\newtheorem{lem}{Lemma}

\begin{document}

\title[Three self-similar solutions]{%
  Three self-similar solutions of Yang-Mills equations in high odd dimensions%
}

\author{Piotr Bizo\'n}
\affiliation{%
  Institute of Theoretical Physics,
  Jagiellonian University,
  Krak\'ow, Poland%
}
\email{piotr.bizon@uj.edu.pl}

\author{Irfan Glogi\'c}
\affiliation{%
  Fakult\"at f\"ur Mathematik,
  Universit\"at Bielefeld,
  Bielefeld, Germany%
}
\email{irfan.glogic@uni-bielefeld.de}

\author{Arthur Wasserman}
\affiliation{%
  Department of Mathematics,
  University of Michigan,
  Ann Arbor, Michigan, USA%
}
\email{awass@umich.edu}

\date{\today}

\begin{abstract}
We consider spherically symmetric Yang-Mills equations with gauge group
$SO(d)$ in $d+1$-dimensional Minkowski spacetime.
For any given odd $d\geq 11$, we establish existence and uniqueness
(modulo reflection symmetry) of exactly $N$ smooth self-similar solutions,
where $N$ is the number of zeros of an explicit polynomial $P_m(z)$
of degree $m=(d-5)/2$ in the interval $0<z<1$.
The number $N$ can be determined algorithmically by an explicit computation.
Our extensive computations for large odd dimensions suggest that $N=3$
for all odd $d\ge 11$.
Two of these self-similar solutions admit closed-form expressions:
one has been known previously, while the other appears to be new.
Our result points toward a relatively simple landscape of possible blowup
scenarios for high-dimensional Yang-Mills equations.
Beyond its purely mathematical interest, this rigidity of self-similar blowup
may also be relevant from a physical perspective, as it constrains the
possible ultraviolet dynamics of non-abelian gauge fields in higher-dimensional
Yang--Mills theories arising in string-inspired extra-dimensional setups
and in holographic models.
\end{abstract}

\maketitle

\section{Introduction}\label{sec:intro}
We consider Yang-Mills (YM) equations with gauge group $SO(d)$
in $d+1$-dimensional Minkowski spacetime
\begin{align}
\partial_{\alpha} F^{\alpha\beta}
  + [A_{\alpha},F^{\alpha\beta}] &= 0,
\label{ym}
\end{align}
where the YM potentials $A_{\alpha}$ (for $\alpha = 0,\dots,d$)
are skew-symmetric $d\times d$ matrices and the YM curvature is
\begin{align}
F_{\alpha\beta}
  &= \partial_{\alpha} A_{\beta}
     - \partial_{\beta} A_{\alpha}
     + [ A_{\alpha}, A_{\beta} ].
\end{align}
We assume the spherically symmetric magnetic ansatz \cite{D}
\begin{align}
A_{\alpha}^{ij}
  = (\delta_{\alpha}^i x^j - \delta_{\alpha}^j x^i)\, w(t,r),
\label{spherical}
\end{align}
where $i,j = 1,\dots,d$ and $r = |\vec{x}|$.
Substituting this ansatz into the YM equations \eqref{ym}
and letting $w(t,r) = \dfrac{1-\phi(t,r)}{r^2}$,
we obtain the semilinear wave equation
\begin{align} \label{ymd}
\phi_{tt}
  = \phi_{rr}
     + \frac{d-3}{r}\, \phi_r
     + \frac{d-2}{r^2}\,\phi(1-\phi^2).
\end{align}
The associated conserved energy is
\begin{align}
E
  &= \int_0^{\infty}
     \left(
       \phi_t^2 + \phi_r^2
       + \frac{d-2}{r^2} (1-\phi^2)^2
     \right)
     r^{d-3}\, dr.
\label{energy}
\end{align}
The basic question for Eq.~\eqref{ymd} is whether solutions starting
from smooth finite-energy initial data can develop singularities
(``blow up'') in finite time and, if so, how the blowup occurs.
A key feature relevant to this issue is the scale invariance
of Eq.~\eqref{ymd} under the transformation
\begin{align}
\phi(t,r)
  &\mapsto \phi_{\lambda}(t,r)
    := \phi\!\left(\frac{t}{\lambda}, \frac{r}{\lambda}\right),
\label{ssscale}
\end{align}
where $\lambda > 0$ is a constant.
Under the scaling, the energy transforms as
\begin{align}
E(\phi_\lambda)
  &= \lambda^{d-4} E(\phi),
\end{align}
which implies that Eq.~\eqref{ymd} is energy-subcritical for $d=3$,
energy--critical for $d=4$, and energy--supercritical for $d\ge 5$.

In the subcritical case $d=3$, concentration to small scales
is energetically penalized, in the sense that rescaling a localized
configuration to smaller length scales increases the energy.
This heuristic is supported by rigorous proofs of global regularity
for smooth finite-energy data \cite{EM,KM}.

In the critical case $d=4$, the energy is invariant under scaling,
so it does not favour dispersion or concentration.
Numerical \cite{BT} and analytical \cite{BOS,RR} results show that
blowup proceeds via concentration of a suitably rescaled instanton;
the scale parameter shrinks to zero while the profile converges
(after appropriate modulation) to the static instanton.

In the supercritical case $d \ge 5$, concentration to small scales
is energetically cheap, so the dynamics favour the formation of
increasingly localized structures and one expects blowup for large
classes of data.
This expectation is borne out by constructions of self-similar
solutions, i.e. solutions of the form
\begin{align} \label{ss}
\phi(t,r) = u(y),\quad
y = \frac{r}{T-t},
\end{align}
where $T$ is a positive constant.
Substituting this ansatz into Eq.~\eqref{ymd} one obtains the
ordinary differential equation
\begin{align}
y^{2} (1-y^{2}) u''
  + \left[(d-3) y - 2y^{3}\right] u'
  + (d-2) u (1-u^2) &= 0\,.
\label{ode}
\end{align}
If the profile function $u(y)$ is smooth on the interval
$0 \leq y \leq 1$ (which corresponds to the interior of the past
light cone of the point $(t=T,r=0)$), then the associated
self-similar solution develops a singularity at time~$T$
from regular initial data.
A key role is played by the explicit self-similar solution
\begin{align}
u_{+}(y)
  &= 1 - \frac{\alpha_{+} y^{2}}{y^{2}+\beta_{+}},
\label{exact+}
\end{align}
where
\begin{align}
\alpha_{+}
  &= 2 + 2\sqrt{\frac{d-4}{3(d-2)}},\nonumber\\
\beta_{+}
  &= \frac{2}{3}(d-4)
     + \frac{1}{3}\sqrt{3(d-2)(d-4)}\nonumber\,.
\end{align}
This solution is regular and nontrivial for all $d\geq 5$;
it was first found for $d=5$ in~\cite{BT} and later generalized
to all $d\geq 5$ in~\cite{BB}.
Its nonlinear stability was proved by the second author \cite{G2,G3},
building on earlier work by the first author~\cite{B2} and by
Donninger et al.~\cite{Do,CDGH}
(see also recent extensions of this result beyond the past light
cone \cite{DO1,DO2}).
These stability results constitute an important step towards showing
that $u_{+}$ governs generic spherically-symmetric blowup, as
conjectured in \cite{B} and supported by numerical computations
\cite{BT,BB}.

Apart from the explicit solution $u_{+}$, other solutions of Eq.~\eqref{ode} have, to the best of our knowledge, been studied only in the low supercritical odd dimensions $5,7,9$. In \cite{B} it was shown, by means of a shooting argument, that there exists a countable family of self-similar solutions $u_{n}(y)$, indexed by their nodal number $n=0,1,\dots$ (the proof was given for $d=5$, but it can be repeated with minor modifications for $d=7$ and $9$). 
 These solutions satisfy $u_n(1)=0$. For $d=5$, $u_0$ coincides with $u_{+}$; however for $d=7$ and $9$, $u_{+}(1)\neq 0$, so $u_{+}$ does not belong to the nodal family $u_n$.
 Earlier, an alternative variational construction of the solution $u_{0}$ in dimensions $d=5,7,9$ was given in \cite{CSS}.  Numerical analysis in $d=5$ showed that the solutions $u_{n}$  have exactly $n$ unstable modes in the associated linearised problem~\cite{B} and  that the solution $u_{1}$ plays the role of a critical solution whose codimension-one stable manifold separates blowup from dispersion \cite{BT}.
This study revealed striking analogies between Yang--Mills theory
in $d=5$ and the Einstein--massless scalar system in the physical
dimension $d=3$, in particular with regard to critical phenomena
at the threshold of black hole formation.

We are not aware of any studies of solutions of Eq.~\eqref{ode} in higher dimensions.
In this paper we begin to fill this gap in the case of odd dimensions $d\geq 11$. The case of even dimensions requires different methods, as will become clear in the course of the analysis below.

Beyond its purely mathematical interest and potential relevance
in string-inspired contexts, our study of high-dimensional
self-similar solutions is motivated by the idea that critical
phenomena at the threshold of blowup may simplify as the dimension
increases.
This expectation is reinforced by the remarkable rigidity of
self-similar solutions revealed in this paper.
Conceptually, this is related to work on critical solutions in
the large-$D$ limit of gravity by Emparan and collaborators~\cite{EH},
although the underlying equations and physical interpretations
differ.

The rest of the paper is organized as follows.
Section~\ref{sec:local} is devoted to the analysis of local
solutions of Eq.~\eqref{ode} near the singular endpoints $y=0$
and $y=1$.
Building on the analysis of Cazenave, Shatah, and Tahvildar-Zadeh
\cite{CSS}, we show that a nontrivial local solution is smooth at
$y=1$ if and only if $z=u(1)^{2}$ is a zero of an explicit polynomial
$P_{m}(z)$ of degree $m=(d-5)/2$.
In Section~\ref{sec:shooting} we establish global bounds and
monotonicity properties of the shooting map, and thereby prove
our main result: the number of nontrivial global smooth profiles
equals the number $N$ of zeros of the polynomial $P_{m}(z)$ in the
interval $0<z<1$.
Explicit computations show that $N=3$ for $m=3,\dots,15$, and we
conjecture that this remains true for all $m\ge 3$, with heuristic
evidence presented in the appendix.
The corresponding three smooth solutions consist of the known
explicit profile $u_{+}$ and two new ones, denoted $u_{-}$ and
$u_{\ast}$.
We find $u_{-}$ in closed form, in close analogy to $u_{+}$, and
construct $u_{\ast}$ numerically.
The extension of these solutions outside the past light cone is
discussed in Section~\ref{sec:outside}.
Finally, in Section~\ref{sec:outlook} we mention the ongoing work
on the stability of the solutions $u_{-}$ and $u_{\ast}$ and their
expected dynamical role.

\section{Local solutions near the origin and the past light cone}
\label{sec:local}
We seek regular solutions of Eq.~\eqref{ode} on the interval
$0\leq y\leq 1$.
The first step in solving this nonlinear boundary-value problem
is to analyze the behavior of solutions near the singular endpoints
$y=0$ and $y=1$.
By reflection symmetry $u \mapsto -u$, it suffices to consider
solutions with $u(0)\geq 0$.
Near $y=0$ it is routine to show that regular solutions behave as
\be\label{taylor0}
u(y) = 1 - a y^{2} + \mathcal{O}(y^{4}),
\ee
 where $a$ is a free
parameter.
These local solutions, denoted below by $u(a,y)$, are analytic in
$a$ and $y$ near $y=0$.

The behavior of regular solutions near $y=1$ is more subtle and
depends on the parity of $d$, as was first observed by Cazenave,
Shatah, and Tahvildar-Zadeh (see Proposition~1 in \cite{CSS}).
The key point of their above analysis is that, in odd dimensions,
smoothness at $y=1$ is not automatic but imposes an algebraic
restriction on the local data.
Near $y=1$ one looks for a local solution in the form
\be\label{taylor}
u(y)
  = c + \sum_{n=1}^{\infty} c_n (y-1)^n,
\ee
with $c=u(1)$ as the leading coefficient.
Substituting this ansatz into Eq.~\eqref{ode} and collecting
coefficients of $(y-1)^n$ yields an infinite algebraic system
in the triangular form
%\begin{subequations}\label{cn}
\begin{align}\label{triang}
(d-5)c_1 + (d-2)c(1-c^{2}) &= 0,\nonumber\\
(d-7)c_2
  + \tfrac12\bigl(2d-11-3(d-2)c^{2}\bigr)c_1 &= 0,\nonumber\\
(d-9)c_3
  + \bigl(d-10-(d-2)c^{2}\bigr)c_2
  - 2c_1 - (d-2)c_1^{2}c &= 0,\nonumber\\
&\vdots \nonumber
\end{align}
%\end{subequations}
For $d=2m+5$ ($m=0,1,\dots$) the factor multiplying $c_{m+1}$
in the triangular system vanishes, so $c_{m+1}$ is not fixed by
the recursion.
This resonance at order $m+1$ leaves $c_{m+1}$ free, while the
coefficients $c,c_1,\dots,c_m$ satisfy a closed algebraic system
of $m+1$ equations.
Solving this system recursively, we find that either
$c^{2}(1-c^{2})=0$ (and then $c_n=0$ for all $n$ from $1$ to $m$)
or $c^{2}$ is a positive root of a polynomial $P_m(c^{2})$ of
degree $m$.
The first five polynomials, unique up to normalization, are:
\begin{subequations}\label{P}
\begin{align}
P_1(c^{2}) &= 5c^{2}-1,\\
P_2(c^{2}) &= 196c^{4}-77c^{2}+1,\\
P_3(c^{2})
  &= (225c^{4}-114c^{2}+1)(21c^{2}-1),\label{P3}\\
P_4(c^{2})
  &= (121c^{4}-77c^{2}+4)(11c^{2}-1)(77c^{2}+3),\label{P4}\\
P_5(c^{2})
  &= (8281c^{4}-6266c^{2}+625)\nonumber\\
  &\quad\times(285837c^{6}-11661c^{4}-1989c^{2}-187).  
 \label{P5}
\end{align}
\end{subequations}
Thus, by a purely local analysis near $y=1$ we infer that in odd
dimensions smooth solutions can exist only for isolated values
of $u(1)$.
In particular, the requirement that $c^2$ be a root of $P_m$
expresses smoothness at $y=1$ as a codimension-one condition on
the local data.
Moreover, if $u(1)$ takes one of those admissible values, then
the solution is smooth at $y=1$.
This follows from the fact that, in general, local solutions near
$y=1$ exhibit a polylogarithmic Fuchsian expansion with a resonant
logarithmic term at order $m+1$ (see \cite{K} for an introduction
to Fuchsian methods)

\begin{align} \label{fuchs}
u(y)
  &= c + \sum_{n=1}^{m} c_n (y-1)^n
     + b (y-1)^{m+1} \log(1-y)\nonumber\\
  &\quad + c_{m+1} (y-1)^{m+1} + \dots.
\end{align}
Generically, the coefficient $b$ is nonzero, so the solution is
only $\mathcal{C}^{m}$ at $y=1$.
However, if $u(1)$ satisfies the above algebraic condition for
smoothness, that is, if $c^2$ is a root of $P_m(c^2)$, then the
resonance is cancelled, i.e. $b=0$.
In this case the Fuchsian expansion \eqref{fuchs} reduces to the
Taylor series \eqref{taylor} and the solution is smooth at $y=1$
(see \cite{CSS} for an alternative proof).
\begin{rem}
If $d$ is even, then the above triangular system  can be solved
recursively for any given $c$, so all values of $u(1)$ are
\emph{a priori} admissible.
For this reason the shooting argument given below for odd $d$
does not work for even~$d$.
\end{rem}
\section{Shooting argument}\label{sec:shooting}
In the following, we consider only solutions which are
regular at $y=0$.
To simplify notation, we write $u(y)$ instead of $u(a,y)$ whenever
the dependence on $a$ is not essential. If $a<0$, then the solution $u(a,y)$ is monotonically increasing (actually, it diverges  before reaching $y=1$), so henceforth we assume that $a\geq 0$.

First, we show that solutions remain bounded on the entire
interval $[0,1]$.
This can be proven using the Lyapunov functional
\begin{align}\label{H}
H(y)
  &= \frac{1}{2} y^{2} (1-y^{2}) u'(y)^2
     - \frac{d-2}{4}\, (1-u(y)^2)^2,
\end{align}
which satisfies
\begin{align}
\frac{dH}{dy}
  &= -(d-4) y\, u'(y)^2.
\end{align}
Since $H(0)=0$ for regular solutions, it follows that $H(y)<0$,
which implies that $|u(y)|<1$ and $H > -\frac{d-2}{4}$, hence
\[
y\sqrt{1-y^{2}}\, |u'|
  < \frac{\sqrt{d-2}}{2}.
\]
It follows that $u'(y)$ and $u(y)$ have finite limits at $y=1$.

\begin{lem}
If $d\geq 10$, then $u(y)$ is monotone decreasing from $u(0)=1$
to $u(1)>0$.
\end{lem}
\begin{proof}
Let $h(y) = y^{3} u'(y)$.
Differentiating Eq.~\eqref{ode}, we obtain
\begin{align}\label{eqh}
y^{2} (1-y^{2}) h''
  + (d-7)y h'
  - \left(3(d-2) u(y)^2 + d-10 \right) h &= 0.
\end{align}
Since $h(y)\sim -2 a y^{4}$ near $y=0$, it is negative and
decreasing for small $y$.
Suppose, for contradiction, that $y_0<1$ is the first point where
$h'(y_0)=0$.
Evaluating Eq.~\eqref{eqh} at $y=y_0$ we get
\begin{align}
y_0^2 (1-y_0^2) h''(y_0)
  &= \left(3(d-2) u(y_0)^{2} + d-10 \right) h(y_0) < 0,\nonumber
\end{align}
which contradicts the assumption that such a point $y_0$ exists.
Thus, $h(y)<0$ and hence $u'(y)<0$ on $(0,1]$.
Evaluating Eq.~\eqref{ode} at~$y=1$, we obtain
\begin{align}\label{odey1}
(d-5) u'(1) + (d-2) u(1) (1-u(1)^{2}) &= 0\,.
\end{align}
Since $u'(1)\!<\!0$ and $u(1)^{2}\!<\!1$, it follows that $u(1)\!>\!0$.
\end{proof}
It follows from Lemma~1 that a general solution $u(a,y)$ admits
the Fuchsian expansion \eqref{fuchs} at $y=1$ with $c>0$.
This defines a $\mathcal{C}^{m-1}$ map $a \mapsto c(a)$ from
$(0,\infty)$ to $(0,1]$.
\begin{lem}
If $d\geq 10$, then the function $c(a)$ is monotone decreasing
from $c(0)=1$ to $c(\infty)=0$.
\end{lem}

\begin{proof}
Obviously $c(0)=1$.
To analyze the limit $a\ra\infty$, we introduce the rescaled
variable $e^{\tau}=\sqrt{a}\, y$.
Substituting $u(y)=U(\tau)$ into Eq.~\eqref{ode} and taking the
limit $a\ra\infty$, we obtain the autonomous equation
\begin{align}\label{odelimit}
U'' + (d-4) U' + (d-2) U (1-U^2)&= 0\,.
\end{align}
By an elementary phase plane analysis we find that solutions
starting at $(U,U') = (1,0)$ at $\tau=-\infty$ tend monotonically
to $(0,0)$ for $\tau\ra \infty$.
This proves that $c(\infty)=0$.

Next, let $v(y) = y^{2} \dfrac{\partial u(a,y)}{\partial a}$.
Differentiating Eq.~\eqref{ode} with respect to $a$, we obtain
\begin{align}\label{eqv}
&y^{2} (1-y^{2}) v''
  + \left((d-7)y+2y^{3}\right) v' \nonumber\\
  &- \left(3(d-2) u(y)^2 + 2 y^2 + d-10 \right) v = 0.
\end{align}
Since $v(y)\sim -y^{4}$ near $y=0$, it is negative and decreasing
for small $y$.
By an argument analogous to that used in the proof of Lemma~1,
we conclude that $v(y)<0$ on $(0,1]$, hence $c(a)$ is monotone
decreasing.
\end{proof}
\begin{rem}
For $5\leq d\leq 9$, the assertions of Lemmas~1 and~2 are false
because the solutions $u(a,y)$ can oscillate around zero, and the
function $c(a)$ is not monotone (see \cite{B} for the proof of
existence of infinitely many smooth oscillating solutions in
$d=5$).
\end{rem}
\begin{rem}
If $d$ in Eq.~\eqref{ode} were treated as a parameter rather than
geometric dimension, the lower bound in Lemmas~1 and~2 could be
strengthened to $d > 6+2\sqrt{3}\approx 9.4641$ by repeating the
proofs with the functions $h(y) = y^{2+\sqrt{3}} u'(y)$ and
$v(y) = y^{1+\sqrt{3}} \frac{\partial u}{\partial a}(a, y)$.
\end{rem}

Combining Lemmas~1 and~2 with the analysis in
Section~\ref{sec:local} of the behavior of smooth solutions at
$y=1$ for odd $d$, we obtain the following rigidity result.
\begin{theorem*}\label{thm:rigidity}
Let $d = 2m+5$ for integer $m\geq 3$.
Then, up to the symmetry $u \mapsto -u$, the number of nontrivial
smooth solutions of Eq.~\eqref{ode} on the interval
$0\leq y\leq 1$ is equal to the number of zeros of the polynomial
$P_m(c^{2})$ in the interval $0 < c^{2} < 1$.
\end{theorem*}
Thus, the nonlinear boundary value problem is transformed into a
finite-dimensional root-counting problem.
This result gives not only a classification theorem, but also a
concrete computational framework for constructing self-similar
solutions.

Let $N$ denote the number of zeroes of the polynomial $P_m(c^2)$.
From the expressions \eqref{P3}--\eqref{P5}, we see explicitly
that $N=3$ for $m=3,4,5$.
By extensive computations we extended this result up to $m=15$,
the upper bound being merely limited by the extent of our
computations.
In the appendix we give a heuristic argument indicating that
$N=3$ for all odd $d\ge 11$; a rigorous proof of this conjecture
remains an open problem.

\vskip 0.1cm
Two of these solutions admit closed form expressions.
The solution $u_{+}$ was given above in \eqref{exact+}.
We find that the second solution has a similar form
\begin{align}
u_{-}(y)
  &= 1 - \frac{\alpha_{-} y^{2}}{y^{2}+\beta_{-}},
\label{exact-}
\end{align}
where
\begin{align}
\alpha_{-}
  &= 2 - 2\sqrt{\frac{d-4}{3(d-2)}},\nonumber\\
\beta_{-}
  &= \frac{2}{3}(d-4)
     - \frac{1}{3}\sqrt{3(d-2)(d-4)}\nonumber\,.
\end{align}
To the best of our knowledge, the solution $u_{-}$ has not
previously appeared in the literature.
It is nontrivial and regular for $d \ge 11$; for $d=10$ one has
$u_{-}=0$, whereas for $5 \le d \le 9$ the denominator
$y^{2} + \beta_{-}$ in \eqref{exact-} has a zero in $(0,1)$, so
$u_{-}$ is singular.

The third solution, denoted by $u_{\ast}$, can be easily
constructed numerically using a shooting method by integrating
local regular solutions $u(a,y)$ from $y=0$ toward $y=1$ and
adjusting the parameter $a$ so that $u(a,1)$ attains the third
admissible value~$c_{*}$ for which $P_m(c_{\ast}^2)=0$.

\begin{figure}[h!]
  \includegraphics[width=0.48\linewidth,
                   trim={2cm 12.5cm 3cm 1.5cm},clip]{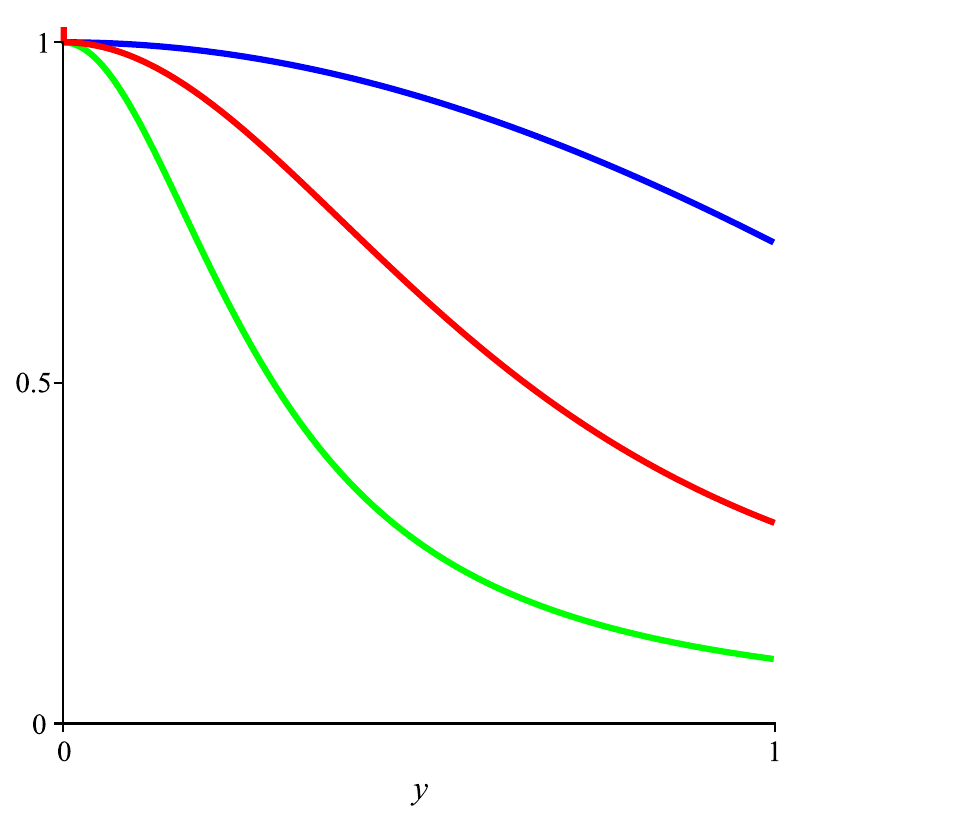}
  \includegraphics[width=0.48\linewidth,
                   trim={2cm 12.5cm 3cm 1.5cm},clip]{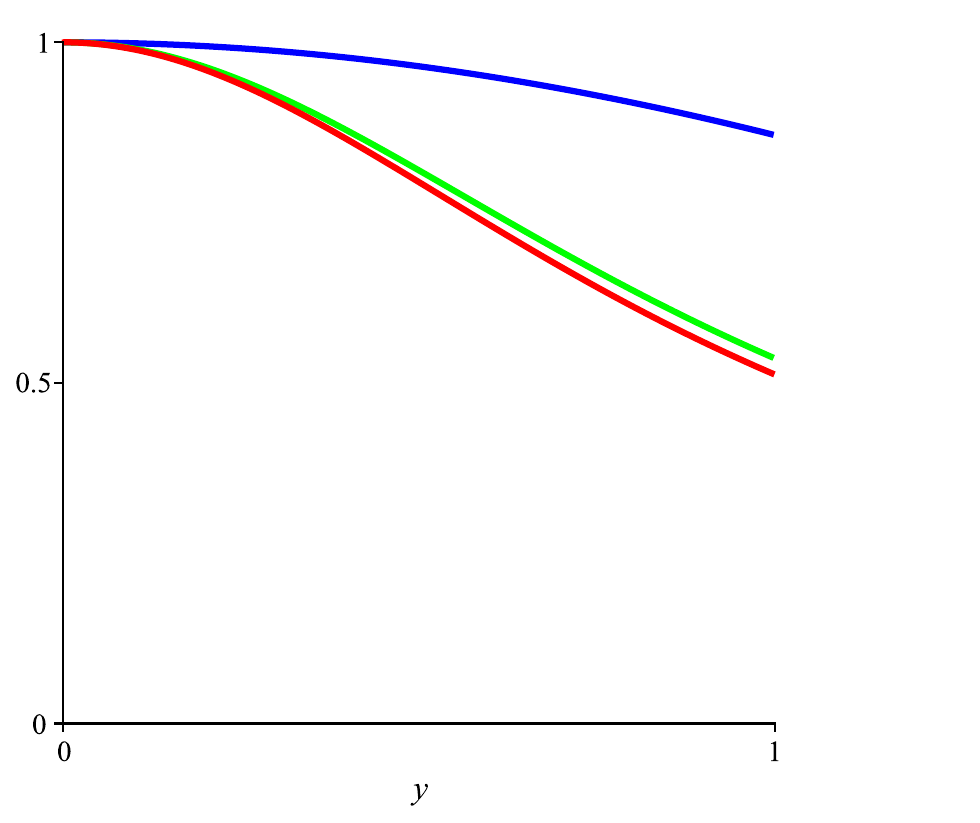}
  \caption{%
    Profiles of $u_{+}$ (blue), $u_{-}$ (green),
    and $u_{\ast}$ (red) in $d=11$ (left) and $d=21$ (right).%
  }
%  \label{fig:twopanels}
\end{figure}
\section{Behavior outside the past light cone}\label{sec:outside}
Let $x = 1/y$.
This coordinate covers the whole spacetime: the past and future
light cones of the point $(t=T,r=0)$ are located at $x=1$ and
$x=-1$ respectively, while the origin corresponds to $x=\infty$
if $t<T$ and $x=-\infty$ if $t>T$.
In terms of $x$ Eq.~\eqref{ode} takes the form
\begin{align}\label{eqx}
(1-x^2) u''(x)
  + (d-5) x u'(x)
  - (d-2) u(x) (1-u(x)^2) &= 0.
\end{align}
We showed above that this equation has three regular solutions in
the interval $1\leq x<\infty$, corresponding to the interior of
the past light cone.
The two explicit solutions $u_{\pm}(x)$ obviously extend smoothly
to the entire spacetime.
That the third solution $u_{\ast}(x)$ can be extended to
$x\in (-1,1)$ can be easily shown using the functional
\begin{align}\label{Q}
\tilde H(x)
  &= -H(1/x) \nonumber\\
  &= \frac{1}{2} (1-x^2) u'(x)^2
     + \frac{d-2}{4} (1-u(x)^2)^2,
\end{align}
which satisfies $\tilde H'(x) = -(d-4) x u'^2$.
For $x\ra -1^{+}$, in analogy to \eqref{fuchs}, we have
\begin{align}\label{fuchs2}
u_{\ast}(x)
  &= \tilde{c}
     + \sum_{n=1}^{m} \tilde c_n (x+1)^n
     + \tilde b (x+1)^{m+1} \log(x+1) \nonumber\\
  &\quad + \tilde c_{m+1} (x+1)^{m+1}+ \dots,
\end{align}
with a generically nonzero coefficient $\tilde b$, hence
$u_{\ast}(x)$ is only $\mathcal{C}^{m}$ at $x=-1$.
\section{Outlook}\label{sec:outlook}
Each smooth self-similar solution of Eq.~\eqref{ymd} provides an
example of finite-time singularity formation from
smooth initial data.
The extent to which a given self-similar solution participates in
dynamics depends on its stability.

As discussed in the introduction, the solution $u_{+}$ has been
proven to be stable and has been shown numerically to govern
generic blowup.
The stability properties of the solutions $u_{-}$ and $u_{\ast}$
have not yet been established.
Our ongoing work indicates that the solution $u_{-}$ has two
unstable modes for $11 \le d \le 17$ and one for $d \ge 19$,
whereas the solution $u_{\ast}$ exhibits the opposite behaviour,
with one unstable mode for $11 \le d \le 17$ and two for
$d \ge 19$.
The solution with a single unstable mode (either $u_{\ast}$ or
$u_{-}$, depending on the dimension) is expected to be critical,
in the sense that its codimension-one stable manifold separates
blowup from dispersion.
This is particularly interesting in the case of solution $u_{-}$,
since its explicit form opens the door to a rigorous analysis of
the threshold dynamics.
\vskip 0.1cm
Finally, we point out that for $d\geq 10$ one can construct
non-self-similar blowup solutions, in analogy with the
construction of such solutions for equivariant wave maps in
$d\geq 7$ \cite{GIN}.
For $d\geq 11$ these so-called type II blowup solutions have the
form $Q\left(r/\lambda(t)\right)$, where $Q(r)$ is a static
solution of Eq.~\eqref{ymd} and
$\lambda(t)\sim (T-t)^{1+\gamma}$.
The anomalous exponent $\gamma>0$ can be determined by the method
of matched asymptotics.
The coexistence of type I and II blowup scenarios for
$d\geq 10$ may lead to new blowup dynamics that are absent in
lower dimensions.
\vskip 0.1cm
\emph{Acknowledgements.}
The research of P.B. was supported in part by the Polish National
Science Center grant no.~2017/26/A/ST2/00530.
The research of I.G. was funded in whole or in part by the
Austrian Science Fund (FWF) 10.55776/PAT5825523.

\section*{Appendix}\label{sec:appendix}
To simplify the notation, we introduce the variable $z=c^2$.
In order to count the number of zeros of the polynomials $P_m(z)$
for $m\geq 3$, let us first factor out the two explicit zeros
corresponding to the solutions $u_{\pm}$,
\begin{align}\label{zpm}
z_{\pm}
  &= \left(1-\frac{\alpha_{\pm}}{\beta_{\pm}+1}\right)^2.
\end{align}
Accordingly, we write
\[
P_m(z) = (z-z_{-})(z-z_{+}) S_m(z),
\]
where
\begin{align}\label{Sm}
S_m(z)
  &= s_{m-2} z^{m-2}
     + s_{m-3} z^{m-3}
     + \cdots + s_0.
\end{align}
For any finite $m$, the coefficients $s_0,\dots,s_{m-2}$ can be
computed explicitly in an algorithmic manner.
By convention, we assume that the leading coefficient $s_{m-2}$
is positive.

The key observation, verified for all $3 \le m \le 40$, is that the descending coefficient sequence $s_{m-2}, s_{m-3}, \dots, s_0$ has exactly one sign change: a block of positive coefficients followed by a contiguous  block of negative coefficients~\cite{AS}. Moreover, the leading coefficient $s_{m-2}$ is large enough to ensure that $S_m(1)>0$, while  $s_0<0$ implies immediately that $S_m(0)<0$. Therefore, by Descartes' rule of signs $S_m(z)$ has exactly one positive real zero, which must lie in $(0,1)$; denote it by \(z_\ast\).

We find that $z_{-}<z_{\ast}<z_{+}$ for $3\leq m\leq 6$, while
$z_{\ast}<z_{-}<z_{+}$ for $m\geq 7$.
Notably, this ordering agrees with the ordering of the
corresponding numbers of unstable modes mentioned in the
previous section.

\end{document}